\documentclass[a4paper,10pt,twoside]{my_cfc}

\usepackage{amssymb,bm,mathrsfs,bbm,amscd}
\usepackage[tbtags]{amsmath}
\usepackage{lastpage}
\usepackage{amsmath,amssymb,amsthm}

\newtheorem{theorem}{Theorem}[section]
\newtheorem{lemma}[theorem]{Lemma}
\newtheorem{corollary}[theorem]{Corollary}

\theoremstyle{definition}
\newtheorem{definition}[theorem]{Definition}

\theoremstyle{remark}
\newtheorem{remark}[theorem]{Remark}

\begin{document}

\title{Combined Delta-Nabla Sum Operator\\
in Discrete Fractional Calculus}


\author{Nuno R. O. Bastos\\
       \texttt{nbastos@mat.estv.ipv.pt}\\[4mm]
       {\sl Department of Mathematics, ESTGV}\\
       {\sl Polytechnic Institute of Viseu}\\
       {\sl 3504-510 Viseu, Portugal}
       \and
       Delfim F. M. Torres\footnote{Corresponding author.}\\
       \texttt{delfim@ua.pt}\\[4mm]
       {\sl Department of Mathematics}\\
       {\sl University of Aveiro}\\
       {\sl 3810-193 Aveiro, Portugal}}

\date{}

\maketitle


\begin{abstract}
We introduce a more general discrete fractional operator,
given by convex linear combination of the delta and nabla fractional sums.
Fundamental properties of the new fractional operator are proved. As particular
cases, results on delta and nabla discrete fractional calculus are obtained.
\end{abstract}



\setcounter{page}{1}


\section{Introduction}

The main goal of this note is to introduce a new and more general
fractional sum operator that unify and extend the discrete
fractional operators used in fractional calculus.
Looking to the literature of discrete fractional difference
operators, two approaches are found (see, \textrm{e.g.}, \cite{4,comNuno:Rui:Z}):
one using the $\Delta$ point of view (sometimes called the forward fractional difference approach),
another using the $\nabla$ perspective (sometimes called the backward fractional
difference approach). Here we introduce a new operator, making use of the symbol
${_{\gamma}}\diamondsuit$ (\textrm{cf.} Definition~\ref{diamond}).
When $\gamma = 1$ the ${_{\gamma}}\diamondsuit$ operator is reduced to the $\Delta$ one;
when $\gamma = 0$ the ${_{\gamma}}\diamondsuit$ operator
coincides with the corresponding $\nabla$ fractional sum.

The work is organized as follows. In Section~\ref{sec:2}
we review the basic definitions of the discrete fractional calculus.
Our results are then given in Section~\ref{sec:3}:
we introduce the fractional diamond sum
(Definition~\ref{diamond}) and prove its main properties.
We end with Section~\ref{sec:4} of conclusions and future perspectives.


\section{Preliminaries}
\label{sec:2}

Here we only give a very short introduction to the basic definitions
in discrete fractional calculus. For more on the subject we refer
the reader to \cite{2,3,7}.

We begin by introducing some notation used throughout.
Let $a$ be an arbitrary real number and $b = a + k$
for a certain $k \in\mathbb{N}$ with $k \ge 2$.
Let $\mathbb{T}= \{a, a + 1, \ldots, b\}$.
According with \cite{8}, we define the
factorial function
$$
t^{(n)} = t(t-1)(t-2)\ldots(t-n+1),
\quad n\in\mathbb{N}.
$$
Also in agreement with the same authors \cite{9},
we define
$$
t^{\overline{n}}=t(t+1)(t+2)\ldots(t+n-1),
\quad n\in\mathbb{N},
$$
and $t^{\overline{0}}=1$.
Extending the two above definitions from an integer $n$ to an arbitrary
real number $\alpha$, we have
\begin{equation*}
t^{(\alpha)}=\frac{\Gamma(t+1)}{\Gamma(t+1-\alpha)}
\text{ \ and \ }
t^{\overline{\alpha}}=\frac{\Gamma(t+\alpha)}{\Gamma(t)},
\end{equation*}
where $\Gamma$ is the Euler gamma function.
Throughout the text we shall use the standard notations
$\sigma(s)=s+1$ and $\rho(s)=s-1$
of the time scale calculus in $\mathbb{Z}$ \cite{8,9}.

\begin{definition}[\cite{12}]
\label{delta}
The discrete delta fractional sum operator is defined by
\begin{equation*}
(\Delta_a^{-\alpha}f)(t)
=\frac{1}{\Gamma(\alpha)}\sum_{s=a}^{t-\alpha}(t-\sigma(s))^{(\alpha-1)}f(s),
\end{equation*}
where $\alpha>0$. Here $f$ is defined for $s=a \mod (1)$ and
$\Delta_a^{-\alpha}f$ is defined for $t=(a+\alpha) \mod (1)$.
\end{definition}

\begin{remark}
Given a real number $a$ and $b = a + k$, $k \in \mathbb{N}$,
$\sum_{s=a}^{b} g(s) = g(a) + g(a+1) + \cdots + g(b)$.
\end{remark}

\begin{remark}
Let $\mathbb{N}_t=\{t,t+1,t+2,\ldots\}$.
We note that $\Delta_a^{-\alpha}$ maps functions defined on
$\mathbb{N}_a$ to functions defined on $\mathbb{N}_{a+\alpha}$.
\end{remark}

Analogously to Definition~\ref{delta},
one considers the discrete nabla fractional sum operator:

\begin{definition}[\cite{4}]
\label{nabla}
The discrete nabla fractional sum operator is defined by
\begin{equation*}
(\nabla_a^{-\beta}f)(t)
=\frac{1}{\Gamma(\beta)}\sum_{s=a}^{t}(t-\rho(s))^{\overline{\beta-1}}f(s),
\end{equation*}
where $\beta>0$. Here $f$ is defined for $s=a \mod (1)$ and
$\nabla_a^{-\beta}f$ is defined for $t=a \mod (1)$.
\end{definition}

\begin{remark}
Let $\mathbb{N}_a=\{a,a+1,a+2,\ldots\}$.
The operator $\nabla_a^{-\beta}$ maps functions defined
on $\mathbb{N}_a$ to functions defined on $\mathbb{N}_{a}$.
The fact that $f$ and $\nabla_a^{-\beta} f$ have the same domain,
while $f$ and $\Delta_a^{-\alpha} f$ do not, explains why some
authors prefer the nabla approach.
\end{remark}

The next result gives a relation between the delta fractional sum
and the nabla fractional sum operators.

\begin{lemma}[\cite{4}]
\label{nabladelta}
Let $0\leq m-1<\nu\leq m$, where $m$ denotes an
integer. Let $a$ be a positive integer, and $y(t)$ be defined on
$t \in \mathbb{N}_a=\{a,a+1,a+2,\ldots\}$. The following statement holds:
$\left(\Delta_a^{-\nu} y\right)(t+\nu)=\left(\nabla_a^{-\nu} y\right)(t)$,
$t\in \mathbb{N}_{a}$.
\end{lemma}


\section{Main Results}
\label{sec:3}

We introduce a general discrete diamond-gamma fractional sum operator
by using a convex combination of the delta
and nabla fractional sum operators.

\begin{definition}
\label{diamond}
The diamond-$\gamma$ fractional operator
of order $(\alpha,\beta)$ is given,
when applied to a function $f$ at point $t$, by
\begin{equation*}
\left(_{\gamma}\diamondsuit_a^{-\alpha,-\beta}f\right)(t)
=\gamma \left(\Delta_a^{-\alpha}f\right)(t+\alpha)
+(1-\gamma) \left(\nabla_a^{-\beta}f\right)(t),
\end{equation*}
where $\alpha>0$, $\beta>0$, and $\gamma\in[0,1]$.
Here, both $f$ and $_{\gamma}\diamondsuit_a^{-\alpha,-\beta}f$
are defined for $t=a \mod (1)$.
\end{definition}

\begin{remark}
Similarly to the nabla fractional operator,
our operator $_{\gamma}\diamondsuit_a^{-\alpha,-\beta}$
maps functions defined on $\mathbb{N}_a$ to functions defined on $\mathbb{N}_a$,
$\mathbb{N}_a=\{a,a+1,a+2,\ldots\}$ for $a$ a given real number.
\end{remark}

\begin{remark}
The new diamond fractional operator of Definition~\ref{diamond} gives,
as particular cases, the operator of Definition~\ref{delta}
for $\gamma=1$,
$$
\left(_{1}\diamondsuit_a^{-\alpha,-\beta}f\right)(t)
=\left(\Delta_a^{-\alpha}f\right)(t+\alpha),
\quad t\equiv a \mod (1),
$$
and the operator of Definition~\ref{nabla} for $\gamma=0$,
$$
\left(_{0}\diamondsuit_a^{-\alpha,-\beta}f\right)(t)
=\left(\nabla_a^{-\beta}f\right)(t),
\quad t\equiv a \mod (1).
$$
\end{remark}

The next theorems give important properties of the
new, more general, discrete fractional operator
$_{\gamma}\diamondsuit_a^{-\alpha,-\beta}$.

\begin{theorem}
\label{thm:SUM}
Let $f$ and $g$ be real functions defined on $\mathbb{N}_a$,
$\mathbb{N}_a=\{a,a+1,a+2,\ldots\}$ for $a$ a given real number.
The following equality holds:
\begin{equation*}
\left(_{\gamma}\diamondsuit_a^{-\alpha,-\beta} (f+g)\right)(t)
=\left({_{\gamma}\diamondsuit}_a^{-\alpha,-\beta}f\right)(t)
+\left({_{\gamma}\diamondsuit}_a^{-\alpha,-\beta}g\right)(t).
\end{equation*}
\end{theorem}

\begin{proof}
The intended equality follows from the definition
of diamond-$\gamma$ fractional sum of order $(\alpha,\beta)$:
\begin{equation*}
\begin{split}
(_{\gamma}\diamondsuit_a^{-\alpha,-\beta}(f+g))(t)
&=\gamma(\Delta_a^{-\alpha}(f+g))(t+\alpha)
+(1-\gamma)(\nabla_a^{-\beta}(f+g))(t)\\
&=\frac{\gamma}{\Gamma(\alpha)}\sum_{s=a}^{t}(t+\alpha-\sigma(s))^{(\alpha-1)}(f(s)
+g(s))+\frac{1-\gamma}{\Gamma(\beta)}\sum_{s=a}^{t}(t-\rho(s))^{\overline{\beta-1}}(f(s)+g(s))\\
&=\frac{\gamma}{\Gamma(\alpha)}\sum_{s=a}^{t}(t+\alpha-\sigma(s))^{(\alpha-1)}f(s)
+\frac{\gamma}{\Gamma(\alpha)}\sum_{s=a}^{t}(t+\alpha-\sigma(s))^{(\alpha-1)}g(s)\\
&\qquad +\frac{1-\gamma}{\Gamma(\beta)}\sum_{s=a}^{t}(t-\rho(s))^{\overline{\beta-1}}f(s)
+\frac{1-\gamma}{\Gamma(\beta)}\sum_{s=a}^{t}(t-\rho(s))^{\overline{\beta-1}}g(s)\\
&= \left[\frac{\gamma}{\Gamma(\alpha)}\sum_{s=a}^{t}(t+\alpha-\sigma(s))^{(\alpha-1)}f(s)
+\frac{1-\gamma}{\Gamma(\beta)}\sum_{s=a}^{t}(t-\rho(s))^{\overline{\beta-1}}f(s)\right]\\
&\qquad + \left[\frac{\gamma}{\Gamma(\alpha)}\sum_{s=a}^{t}(t+\alpha-\sigma(s))^{(\alpha-1)}g(s)
+\frac{1-\gamma}{\Gamma(\beta)}\sum_{s=a}^{t}(t-\rho(s))^{\overline{\beta-1}}g(s)\right]\\
&=(_{\gamma}\diamondsuit_a^{-\alpha,-\beta}f)(t) + (_{\gamma}\diamondsuit_a^{-\alpha,-\beta}g)(t).
\end{split}
\end{equation*}
\end{proof}

\begin{theorem}
\label{thm:const}
Let $f(t)=k$ on $\mathbb{N}_a$, $k$ a constant.
The following equality holds:
$$
(_{\gamma}\diamondsuit_a^{-\alpha,-\beta}f)(t)
=\gamma\frac{\Gamma(t-a+1+\alpha) k}{\Gamma(\alpha+1)\Gamma(t-a+1)}
+(1-\gamma)\frac{\Gamma(t-a+1+\beta) k}{\Gamma(\beta+1)\Gamma(t-a+1)}\, .
$$
\end{theorem}

\begin{proof}
By definition of diamond-$\gamma$ fractional
sum of order $(\alpha,\beta)$, we have
\[\begin{split}
(_{\gamma}&\diamondsuit_a^{-\alpha,-\beta}k)(t)
=\gamma(\Delta_a^{-\alpha}k)(t+\alpha)+(1-\gamma)(\nabla_a^{-\beta}k)(t)
=\frac{\gamma}{\Gamma(\alpha)}\sum_{s=0}^t
k(t+\alpha-\sigma(s))^{(\alpha-1)}
+\frac{1-\gamma}{\Gamma(\beta)}\sum_{s=0}^t
k(t-\rho(s))^{\overline{\beta-1}}\\
&=\gamma\frac{\Gamma(t-a+1+\alpha)}{\alpha\Gamma(\alpha)\Gamma(t-a+1)}k+
(1-\gamma)\frac{\Gamma(t-a+1+\beta)}{\beta\Gamma(\beta)\Gamma(t-a+1)}k
=\gamma\frac{\Gamma(t-a+1+\alpha)}{\Gamma(\alpha+1)\Gamma(t-a+1)}k+
(1-\gamma)\frac{\Gamma(t-a+1+\beta)}{\Gamma(\beta+1)\Gamma(t-a+1)}k.
\end{split}
\]
\end{proof}

\begin{corollary}
\label{m:f:cor}
Let $f(t) \equiv k$
for a certain constant $k$. Then,
\begin{equation}
\label{Miller_Constant}
(\Delta_a^{-\alpha}f)(t+\alpha)
=\frac{\Gamma(t-a+1+\alpha)}{\Gamma(\alpha+1)\Gamma(t-a+1)}k.
\end{equation}
\end{corollary}

\begin{proof}
The result follows from Theorem~\ref{thm:const}
choosing $\gamma=1$ and recalling that
$(_{1}\diamondsuit_a^{-\alpha,-\beta}k)(t)
= (\Delta_a^{-\alpha}k)(t+\alpha)$.
\end{proof}

\begin{remark}
In the particular case when $a=0$, equality
\eqref{Miller_Constant} coincides with the result
of \cite[Sect.~5]{12}.
\end{remark}

The fractional nabla result analogous to Corollary~\ref{m:f:cor}
is easily obtained:

\begin{corollary}
If $k$ is a constant, then
\begin{equation*}
(\nabla_a^{-\beta}k)(t)
=\frac{\Gamma(t-a+1+\beta)}{\Gamma(\beta+1)\Gamma(t-a+1)}k.
\end{equation*}
\end{corollary}

\begin{proof}
The result follows from Theorem~\ref{thm:const}
choosing $\gamma=0$ and recalling that
$(_{0}\diamondsuit_a^{-\alpha,-\beta}k)(t)
=(\nabla_a^{-\beta}k)(t)$.
\end{proof}

\begin{theorem}
\label{thm:10}
Let $f$ be a real valued function and
$\alpha_1$, $\alpha_2$, $\beta_1$, $\beta_2>0$. Then,
\begin{equation*}
\left(_{\gamma}\diamondsuit_a^{-\alpha_1,
-\beta_1}\left(_{\gamma}\diamondsuit_a^{-\alpha_2,-\beta_2}f\right)\right)(t)
=\gamma\left({_{\gamma}}\diamondsuit_a^{-(\alpha_1+\alpha_2),
-(\beta_1+\alpha_2)} f\right)(t)+(1-\gamma)\left(
{_{\gamma}}\diamondsuit_a^{-(\alpha_1+\beta_2),-(\beta_1+\beta_2)} f\right)(t).
\end{equation*}
\end{theorem}

\begin{proof}
Direct calculations show the intended relation:
\begin{equation*}
\begin{split}
(_{\gamma}&\diamondsuit_a^{-\alpha_1,-\beta_1}(_{\gamma}\diamondsuit_a^{-\alpha_2,-\beta_2}f))(t)
=\gamma(\Delta_a^{-\alpha_1}(_{\gamma}\diamondsuit_a^{-\alpha_2,-\beta_2}f))(t+\alpha_1)
+ (1-\gamma)(\nabla_a^{-\beta_1}(_{\gamma}\diamondsuit_a^{-\alpha_2,-\beta_2}f))(t)\\
&=\gamma^2 (\Delta_a^{-\alpha_1}\left(\Delta_a^{-\alpha_2}f\right))(t+\alpha_1+\alpha_2)
+ \gamma(1-\gamma)\left(\Delta_a^{-\alpha_1}\left(\nabla_a^{-\beta_2}f\right)\right)(t+\alpha_1)
+ (1-\gamma)\gamma\left(\nabla_a^{-\beta_1}\left(\Delta_a^{-\alpha_2}f\right)\right)(t+\alpha_2)\\
&\quad +  (1-\gamma)^2\left(\nabla_a^{-\beta_1}\left(\nabla_a^{-\beta_2}f\right)\right)(t)\\
&=\gamma^2\left(\Delta_a^{-(\alpha_1+\alpha_2)}f\right)(t+\alpha_1+\alpha_2)
+ \gamma(1-\gamma)\left(\Delta_a^{-\alpha_1}\left(\Delta_a^{-\beta_2}f\right)\right)(t+\alpha_1+\beta_2)\\
&\quad + (1-\gamma)\gamma\left(\nabla_a^{-\beta_1}\left(\nabla_a^{-\alpha_2}f\right)\right)(t)
+(1-\gamma)^2 \left(\nabla_a^{-(\beta_1+\beta_2)}f\right)(t)\\
&=\gamma^2\left(\Delta_a^{-(\alpha_1+\alpha_2)}f\right)(t+\alpha_1+\alpha_2)
+ \gamma(1-\gamma)\left(\Delta_a^{-(\alpha_1+\beta_2)}f\right)(t+\alpha_1+\beta_2)\\
&\quad +  (1-\gamma)\gamma\left(\nabla_a^{-(\beta_1+\alpha_2)}f\right)(t)
+(1-\gamma)^2\left(\nabla_a^{-(\beta_1+\beta_2)}f\right)(t)\\
&=\gamma\left[\gamma\left(\Delta_a^{-(\alpha_1+\alpha_2)}f\right)(t+\alpha_1+\alpha_2)
+ (1-\gamma)\left(\nabla_a^{-(\beta_1+\alpha_2)}f\right)(t)\right]\\
&\quad +  (1-\gamma)\left[\gamma\left(\Delta_a^{-(\alpha_1+\beta_2)}f\right)(t+\alpha_1+\beta_2)
+ (1-\gamma)\left(\nabla_a^{-(\beta_1+\beta_2)}f\right)(t)\right].
\end{split}
\end{equation*}
\end{proof}

\begin{remark}
If $\gamma=0$, then
$\left(_{\gamma}\diamondsuit_a^{-\alpha_1,-\beta_1}(_{\gamma}\diamondsuit_a^{-\alpha_2,-\beta_2}f)\right)(t)
=\left(\nabla_a^{-(\beta_1+\beta_2)}f\right)(t)$.
\end{remark}

\begin{remark}
If $\gamma=1$, then
$\left(_{\gamma}\diamondsuit_a^{-\alpha_1,-\beta_1}(_{\gamma}\diamondsuit_a^{-\alpha_2,-\beta_2}f)\right)(t)
=\left(\Delta_a^{-(\alpha_1+\alpha_2)}f\right)(t+\alpha_1+\alpha_2)$.
\end{remark}

\begin{remark}
If $\alpha_1=\alpha_2=\alpha$ and $\beta_1=\beta_2=\beta$, then
$\left(_{\gamma}\diamondsuit_a^{-\alpha,-\beta}(_{\gamma}\diamondsuit_a^{-\alpha,-\beta}f)\right)(t)
= \left({_{\gamma}}\diamondsuit_a^{-\alpha,-\beta}f\right)(t)$.
\end{remark}

We now prove a general Leibniz formula.

\begin{theorem}[Leibniz formula]
\label{thm:ProductRule}
Let $f$ and $g$ be real valued functions, $0<\alpha,~\beta<1$.
For all $t$ such that $t=a \mod (1)$, the following equality holds:
\begin{multline}
\label{eq:GLF}
\left(_{\gamma}\diamondsuit_a^{-\alpha,-\beta}(fg)\right)(t)
=\gamma\sum_{k=0}^\infty\binom{-\alpha}{k}\left[\left(\nabla^k
g\right)(t)\right] \cdot\left[\left(\Delta_a^{-(\alpha+k)}f\right)(t+\alpha+k)\right]\\
+(1-\gamma)\sum_{k=0}^\infty\binom{-\beta}{k}\left[\left(\nabla^k
g\right)(t)\right]\left[\left(\Delta_a^{-(\beta+k)}f\right)(t+\beta + k)\right],
\end{multline}
where
$$
\binom{u}{v}=\frac{\Gamma(u+1)}{\Gamma(v+1)\Gamma(u-v+1)}.
$$
\end{theorem}

\begin{proof}
By definition of the diamond fractional sum,
\begin{equation*}
\begin{split}
\left(_{\gamma}\diamondsuit_a^{-\alpha,-\beta}(fg)\right)(t)
&=\gamma \left(\Delta_a^{-\alpha}(fg)\right)(t+\alpha)
+(1-\gamma)\left(\nabla_a^{-\beta}(fg)\right)(t)\\
&=\frac{\gamma}{\Gamma(\alpha)}\sum_{s=a}^{t}(t+\alpha-\sigma(s))^{(\alpha-1)}f(s)g(s)
+\frac{1-\gamma}{\Gamma(\beta)}\sum_{s=a}^{t}(t-\rho(s))^{\overline{\beta-1}}f(s)g(s).
\end{split}
\end{equation*}
By Taylor's expansion of $g(s)$ \cite{1},
$$
g(s)=\sum_{k=0}^\infty \frac{(s-t)^{\overline{k}}}{k!}
(\nabla^k g)(t)=\sum_{k=0}^\infty (-1)^k\frac{(t-s)^{(k)}}{k!}(\nabla^k g)(t).
$$
Substituting the Taylor series of $g(s)$ at $t$,
\begin{multline*}
\left(_{\gamma}\diamondsuit_a^{-\alpha,-\beta}(fg)\right)(t)
=\frac{\gamma}{\Gamma(\alpha)}\sum_{s=a}^{t}(t+\alpha-\sigma(s))^{(\alpha-1)}f(s)
\left[\sum_{k=0}^\infty (-1)^k(t-s)^{(k)}\frac{(\nabla^k g)(t)}{k!}\right]\\
+\frac{1-\gamma}{\Gamma(\beta)}\sum_{s=a}^{t}(t-\rho(s))^{\overline{\beta-1}}f(s)
\left[\sum_{k=0}^\infty
(-1)^k(t-s)^{(k)}\frac{(\nabla^k g)(t)}{k!}\right].
\end{multline*}
Since
\begin{equation*}
\begin{split}
(t+\alpha-\sigma(s))^{(\alpha-1)}(t-s)^{(k)}&=(t+\alpha-\sigma(s))^{(\alpha+k+1)},\\
(t-\rho(s))^{\overline{\beta-1}}(t-s)^{(k)}&=(t+\beta-\sigma(s))^{(\beta+k+1)},
\end{split}
\end{equation*}
and $\displaystyle\sum_{s=t-k+1}^{t}(t-s)^{(k)}=0$, we have
\begin{multline*}
\left(_{\gamma}\diamondsuit_a^{-\alpha,-\beta}(fg)\right)(t)
=\frac{\gamma}{\Gamma(\alpha)}\sum_{k=0}^\infty
(-1)^k\frac{(\nabla^k g)(t)}{k!} \sum_{s=a}^{t-k}(t+\alpha-\sigma(s))^{(\alpha+k-1)}f(s)\\
+\frac{1-\gamma}{\Gamma(\beta)}\sum_{k=0}^\infty
(-1)^k\frac{(\nabla^k g)(t)}{k!}
\sum_{s=a}^{t-k}(t+\beta-\sigma(s))^{(\beta+k-1)}f(s).
\end{multline*}
Because
$$
(-1)^k=\frac{\Gamma(-\alpha+1)\Gamma(\alpha)}{\Gamma(-\alpha+k+1)\Gamma(k+\alpha)}
=\frac{\Gamma(-\beta+1)\Gamma(\beta)}{\Gamma(-\beta+k+1)\Gamma(k+\beta)}
$$
and $k!=\Gamma(k+1)$, the above expression becomes
\begin{equation*}
\begin{split}
\left(_{\gamma}\diamondsuit_a^{-\alpha,-\beta}(fg)\right)(t)
&=\frac{\gamma}{\Gamma(\alpha)} \sum_{k=0}^\infty (\nabla^k g)(t)
\binom{-\alpha}{k}\cdot\left[\frac{1}{\Gamma(k+\alpha)}
\sum_{s=a}^{t-k}(t+\alpha-\sigma(s))^{(\alpha+k-1)}f(s)\right]\\
&\qquad+\frac{1-\gamma}{\Gamma(\beta)} \sum_{k=0}^\infty (\nabla^k g)(t)
\binom{-\beta}{k}\left[\frac{1}{\Gamma(k+\beta)}
\sum_{s=a}^{t-k}(t+\beta-\sigma(s))^{(\beta+k-1)}f(s)\right]\\
&=\gamma\sum_{k=0}^\infty \binom{-\alpha}{k}(\nabla^k g)(t)
(\Delta_a^{-(\alpha+k)}f)(t+\alpha+k)\\
&\qquad+(1-\gamma)\sum_{k=0}^\infty
\binom{-\beta}{k}(\nabla^k g)(t)(\Delta_a^{-(\beta+k)}f)(t+\beta+k).
\end{split}
\end{equation*}
\end{proof}

\begin{remark}
Choosing $\gamma=0$ in our Leibniz formula \eqref{eq:GLF}, we obtain that
$$(\nabla_a^{-\beta}(fg))(t)
=\sum_{k=0}^\infty\binom{-\beta}{k}\left[(\nabla^k
g)(t)\right]\left[(\Delta_a^{-(\beta+k)}f)(t+\beta+k)\right].
$$
\end{remark}

\begin{remark}
Choosing $\gamma=1$ in our Leibniz formula \eqref{eq:GLF}, we obtain that
\begin{equation}
\label{LeibnizDelta}
(\Delta_a^{-\alpha}(fg))(t+\alpha)
=\sum_{k=0}^\infty\binom{-\alpha}{k}\left[(\nabla^k g)(t)\right]
\left[(\Delta_a^{-(\alpha+k)}f)(t+\alpha+k)\right].
\end{equation}
As a particular case of \eqref{LeibnizDelta}, let $a=0$.
Then, recalling Lemma~\ref{nabladelta},
we obtain the Leibniz formulas of \cite{5}.
\end{remark}


\section{Conclusion}
\label{sec:4}

The discrete fractional calculus is a subject under
strong current research
(see, \textrm{e.g.}, \cite{6,7,10,Goodrich1,Goodrich2} and references therein).
Two versions of the discrete fractional calculus, the delta and the nabla,
are now standard in the fractional theory.
Motivated by the diamond-alpha dynamic derivative on time scales \cite{11,13,14}
and the fractional derivative of \cite{withBasia:Spain2010},
we introduce here a combined diamond-gamma fractional sum
of order (alpha, beta), as a linear combination
of the delta and nabla fractional sum operators
of order alpha and beta, respectively.
The new operator interpolates between the delta and nabla cases,
reducing to the standard fractional delta operator when $\gamma =1$
and to the fractional nabla sum when $\gamma =0$.

Using the discrete fractional diamond sum here proposed,
one can now introduce the discrete fractional diamond difference
in the usual way. It is our intention to generalize the new
discrete diamond fractional operator
to an arbitrary time scale $\mathbb{T}$
(\textrm{i.e.}, to an arbitrary nonempty
closed set of the real numbers). Another line of research,
to be addressed elsewhere, consists to investigate
the usefulness of modeling with fractional diamond equations
and study corresponding fractional variational principles.


\newpage

\leftline{\bf\ Acknowledgments}

\vskip 10 pt

This work is part of the first author's Ph.D. project, carried out at
the University of Aveiro under the framework of the Doctoral
Programme \emph{Mathematics and Applications} of Universities of
Aveiro and Minho, and was partially presented during the
\emph{3rd Conference on Nonlinear Science and Complexity} (NSC10),
Cankaya University, Ankara, 26-29 July, 2010.
The financial support of the Polytechnic Institute
of Viseu and \emph{The Portuguese Foundation for Science and
Technology} (FCT), through the ``Programa de apoio \`{a}
forma\c{c}\~{a}o avan\c{c}ada de docentes do Ensino Superior
Polit\'{e}cnico'', Ph.D. fellowship SFRH/PROTEC/49730/2009, is here
gratefully acknowledged. The authors were also supported by FCT through the
\emph{Center for Research and Development in Mathematics and Applications} (CIDMA).

\vskip 20 pt



\end{document}